\documentclass{amsart}
\usepackage{amsmath}
\usepackage{amssymb}
\usepackage{amsthm}
\usepackage{enumerate}
\usepackage{graphics}
\usepackage{float}
\usepackage{caption, subcaption} 

\usepackage{capt-of}
\usepackage{tikz}
\usetikzlibrary{patterns}
\usetikzlibrary{decorations.markings}
\usepackage{multirow}
\usepackage{array}

\newtheorem{theorem}{Theorem}
\newtheorem{prop}[theorem]{Proposition}
\newtheorem{lemma}[theorem]{Lemma}
\newtheorem{rem}[theorem]{Remark}
\newtheorem{exmp}[theorem]{Example}

\usepackage{amssymb,tikz}
\usetikzlibrary{positioning}
\usepackage[colorlinks=false,urlbordercolor=white]{hyperref}
  
\tikzset{sgplattice/.style={inner sep=1pt,norm/.style={red!50!blue},char/.style={blue!50!black}, lin/.style={black!50}},cnj/.style={black!50,yshift=-2.5pt,left=-1pt of #1,scale=0.5,fill=white}}

\begin{document}
\title[$(2^{k-1}-1)$-pyramidal symmetric $(2^k-1,2^{k-1},2^{k-2})$-designs]
{Symmetric $(2^k-1,2^{k-1},2^{k-2})$-designs which are $(2^{k-1}-1)$-pyramidal over abelian groups}
\author{Mark Pankov}
\address{Faculty of Mathematics and Computer Science, 
University of Warmia and Mazury, S{\l}oneczna 54, 10-710 Olsztyn, Poland}
\email{pankov@matman.uwm.edu.pl}

\begin{abstract}
A design is  called $t$-pyramidal when it has an automorphism group which fixes $t$
points and acts sharply transitively on the remaining points. 
We determine all symmetric $(2^k-1,2^{k-1},2^{k-2})$-designs which are $(2^{k-1}-1)$-pyramidal over abelian groups. 
\end{abstract}

\maketitle

\section{Introduction}
Following \cite{BRT} we say that a block design is $t$-{\it pyramidal} 
if it has an automorphism group which leaves $t$ points fixed and acts sharply transitively on the remaining points.
A description of pyramidal  Steiner triple systems  over abelian groups and a list of references for the topic 
can be found  in \cite{CTZ}.

We consider a similar problem for symmetric $(2^k-1,2^{k-1},2^{k-2})$-designs, i.e.
the complements of Hadamard $(4m-1,2m-1, m-1)$-designs with $m=2^{k-2}$.
In this class, we determine all $(2^{k-1}-1)$-pyramidal designs  over abelian groups. 



The point-line geometry  formed by all points of ${\rm PG}(2^k-1,2)$ with Hamming weight $2^{k-1}$ is investigated in \cite{PPZ1}.
Maximal singular subspaces of this geometry correspond to binary simplex codes of dimension $k$ 
and $(2^k-1)$-element cliques of the collinearity graph are the sets of blocks of symmetric $(2^k-1,2^{k-1},2^{k-2})$-designs.
Maximal singular subspaces are simplest examples of such cliques and the corresponding designs are isomorphic to 
the design of points and hyperplane complements of ${\rm PG}(k-1,2)$.

If $O$ is a {\it center block} of a symmetric $(2^k-1,2^{k-1},2^{k-2})$-design ${\mathcal D}$,
i.e. for every block $B$ distinct from $O$ the symmetric difference $O\triangle B$ is a block,
then ${\mathcal D}$ can be decomposed in the sum of symmetric $(2^{k-1}-1,2^{k-2},2^{k-3})$-designs
${\mathcal D}_O$ and ${\mathcal D}_Z$ whose sets of points  are the complement of $O$ and a $(2^{k-1}-1)$-element subset $Z\subset O$,
respectively.  This decomposition is defined for any such $Z$ and the first component ${\mathcal D}_O$ does not depend on $Z$.
If ${\mathcal D}_Z$ is isomorphic to the design of points and hyperplane complements of ${\rm PG}(k-2,2)$ for a certain $Z$,
then the same holds for all others
and we say that ${\mathcal D}$ is the sum of ${\mathcal D}_O$ and the design of points and hyperplane complements of ${\rm PG}(k-2,2)$.
In this case, there is an abelian group of automorphisms of ${\mathcal D}$ generated by $k-1$ involutions
which acts sharply transitively on the center block $O$ and leaves the remaining points fixed, i.e.
${\mathcal D}$ is $(2^{k-1}-1)$-pyramidal over $C^{k-1}_2$ (the direct product of $k-1$ exemplars of the cyclic group $C_2$).

We show that there exist no other symmetric $(2^k-1,2^{k-1},2^{k-2})$-designs which are $(2^{k-1}-1)$-pyramidal over abelian groups.

\section{Point-line geometries and symmetric block designs}
Recall that a {\it point-line geometry} is a non-empty set, whose elements are called {\it points}, together with a distinguished family of subsets called {\it lines}. 
Every line contains at least two distinct  points and the intersection of two distinct lines contains at most one point. 
We say that two distinct points are {\it collinear} if there is a line containing them.
The associated {\it collinearity graph} is the simple graph whose vertices are points
and two vertices are connected by an edge if they are collinear.
{\it Isomorphisms} of point-line geometries are line preserving bijections between the point sets
and an {\it automorphism} of a geometry is an  isomorphism to itself.

Every non-zero binary vector $(x_1,\dots,x_n)$ can be identified with the subset of  $\{1,\dots,n\}$ 
formed by all indices $i$ such that $x_i\ne 0$.
Therefore, we can consider the projective space ${\rm PG}(n-1,2)$  as the set of all non-empty subsets of an $n$-element set $P$
whose lines are triples of type $X,Y,X\triangle Y$, where $X\triangle Y$ is the symmetric difference of $X,Y$.
For every positive integer $m$ satisfying $3m\le n$ we denote by ${\mathcal P}_m(n)$ the subgeometry of ${\rm PG}(n-1,2)$
whose points are $2m$-element subsets of $P$ (points of Hamming weight $2m$)
and whose lines are precisely the lines of ${\rm PG}(n-1,2)$ contained in ${\mathcal P}_m(n)$.
If $3m>n$, then the set of all $2m$-element subsets of $P$ contains no line;
the same holds for any subset of ${\rm PG}(n-1,2)$ formed by all points of fixed odd Hamming weight.
Every permutation on $P$ induces an automorphism of the geometry ${\mathcal P}_m(n)$.
All automorphisms of this geometry are determined in \cite{PPZ1}. 
In some cases, there are automorphisms which are not induced by permutations.
Two $2m$-element subsets $X,Y\subset P$ are collinear points of ${\mathcal P}_m(n)$
if and only if $X\cap Y$ contains precisely $m$ elements.
By Fisher's inequality \cite{MM}, a clique of  the collinearity graph of ${\mathcal P}_m(n)$ contains at most $n$ vertices.
It was observed in \cite{PPZ1,PPZ2} that every $n$-element clique of this graph  is the set of blocks of a symmetric $(n,2m,m)$-design.

A {\it symmetric $(v,k,\lambda)$-design} can be defined as a pair ${\mathcal D}=(P, {\mathcal B})$, 
where $P$ is a set consisting of $v$ elements called {\it points}
and ${\mathcal B}$ is a $v$-element set of $k$-element subsets of $P$ called {\it blocks}
such that the intersection of any two distinct blocks contains precisely $\lambda$ points
(this is equivalent to the standard definition, see, for example, \cite[p. 299]{LM}). 
Then every point is contained in precisely $k$ distinct blocks
and every $2$-element subset of $P$ is contained in precisely $\lambda$ blocks. 
The pair  ${\mathcal D}^*=({\mathcal B}, \{{\mathcal B}_p\}_{p\in P})$, where ${\mathcal B}_p$ is the set of blocks containing $p$, 
is also a symmetric $(v,k,\lambda)$-design called {\it dual} to ${\mathcal D}$.  
{\it Design isomorphisms} are block preserving bijections between the point sets of designs and an {\it automorphism} of a design is an isomorphism to itself. 
A design is said to be {\it self-dual} if it is isomorphic to the dual design. 

So, ${\mathcal D}=(P, {\mathcal B})$, where $P$ is an $n$-element set and ${\mathcal B}$ is a set of subsets of  $P$, is a symmetric $(n,2m,m)$-design if and only if   
${\mathcal B}$ is an $n$-element clique of the collinearity graph of ${\mathcal P}_m(n)$.
We restrict ourself to the case when 
$$m=2^{k-2},\;\;n=4m-1=2^k-1$$
with $k\ge 3$.
Recall that ${\rm PG}(k-1,2)$ contains precisely  $2^k-1$ points and  the same number of hyperplanes. 

A {\it singular subspace} of a point-line geometry is a subset of the set of points, where any two distinct points are collinear 
and the line joining them is contained in this subset. 
By \cite{PPZ1}, every maximal  singular subspace of ${\mathcal P}_{2^{k-2}}(2^k-1)$ is isomorphic to ${\rm PG}(k-1,2)$
(the case $k=2$ is trivial: ${\mathcal P}_1(3)$ is a line consisting of three points) and, consequently,
it is a  $(2^k-1)$-element clique of the collinearity graph of ${\mathcal P}_{2^{k-2}}(2^k-1)$.
Every symmetric $(2^k-1,2^{k-1},2^{k-2})$-design whose blocks form a maximal singular subspace of ${\mathcal P}_{2^{k-2}}(2^k-1)$ is  isomorphic
to the design of points and hyperplane complements of ${\rm PG}(k-1,2)$ \cite[Proposition 1]{PPZ2}.

In the case when $k\in \{2,3\}$, 
every $(2^k-1)$-element clique of the collinearity graph of ${\mathcal P}_{2^{k-2}}(2^k-1)$ is a maximal singular subspace 
and every symmetric $(2^k-1,2^{k-1},2^{k-2})$-design is isomorphic to the design of points and hyperplane complements of ${\rm PG}(k-1,2)$
(for $k\in \{2,3\}$ hyperplanes of ${\rm PG}(k-1,2)$ are points and lines, respectively).

The design of points and hyperplane complements of ${\rm PG}(k-1,2)$ is self-dual:
every isomorphism between ${\rm PG}(k-1,2)$ and the dual projective space (whose points are hyperplanes) induces an isomorphism between this design and its dual.

\begin{prop}\label{prop-ps}
A symmetric $(2^k-1,2^{k-1},2^{k-2})$-design is isomorphic to the design of points and hyperplane complements of ${\rm PG}(k-1,2)$
if and only if for any distinct points $p,q$ there is a point $t$ such that there is no block containing $p,q,t$.
\end{prop}

\begin{proof}
Let ${\mathcal D}=(P,{\mathcal B})$ be a symmetric $(2^k-1,2^{k-1},2^{k-2})$-design
and let ${\mathcal D}^*=({\mathcal B}, \{{\mathcal B}_p\}_{p\in P})$ be the dual design. 
Recall that ${\mathcal B}_p$ is the set of blocks of ${\mathcal D}$ containing $p$. 
The above condition can be reformulated as follows: for any distinct $p,q\in P$ there is $t\in P$ such that 
\begin{equation}\label{eq-line}
{\mathcal B}_p\cap {\mathcal B}_q\cap {\mathcal B}_t=\emptyset.
\end{equation}
Since ${\mathcal B}_p,{\mathcal B}_q, {\mathcal B}_t$ are $2^{k-1}$-element subsets of ${\mathcal B}$ mutually intersecting in $2^{k-2}$-element subsets,
\eqref{eq-line} holds if and only if ${\mathcal B}_t$ is the symmetric difference of ${\mathcal B}_p$ and ${\mathcal B}_q$, in other words, 
${\mathcal B}_p,{\mathcal B}_q, {\mathcal B}_t$ form a line in the geometry ${\mathcal P}_{2^{k-2}}(2^k-1)$. 
Therefore,  for any distinct $p,q\in P$ there is $t\in P$ such that there is no block of ${\mathcal D}$ containing $p,q,t$
if and only if the blocks of ${\mathcal D}^*$ form a maximal singular subspace of ${\mathcal P}_{2^{k-2}}(2^k-1)$
which is equivalent to the fact that ${\mathcal D}^*$ is  isomorphic to the design of points and hyperplane complements of ${\rm PG}(k-1,2)$. 
The latter design is self-dual and  we get the claim.
\end{proof}

As above, we assume that ${\mathcal D}=(P,{\mathcal B})$ is a symmetric $(2^k-1,2^{k-1},2^{k-2})$-design.
Suppose that for some distinct $p,q\in P$ there is $t\in P$ such that there is no block of ${\mathcal D}$ containing $p,q,t$.
It was established above that such $t$ is unique 
and the subset $\{p,q,t\}$ corresponds to a line of ${\mathcal P}_{2^{k-2}}(2^k-1)$ contained in the set of blocks of ${\mathcal D}^*$.
For this reasons, we say that $\{p,q,t\}$ is a {\it line} of ${\mathcal D}$ and $p,q$ are {\it collinear} points of ${\mathcal D}$.

We say also that $p\in P$ is a {\it center point} of ${\mathcal D}$ if it is collinear to the remaining points of ${\mathcal D}$.
Then
$$P\setminus \{p\}=\{q_1,t_1,\dots,q_{2^{k-1}-1},t_{2^{k-1}-1}\},$$
where $\{p,q_i,t_i\}$ is a line of ${\mathcal D}$ for every $i\in \{1,\dots,2^{k-1}-1\}$.
Since blocks do not contain lines, 
every block containing $p$ intersects each $\{q_i,t_i\}$ precisely in one point.
If a block does not contain $p$ and contains one of $q_i,t_i$, then it also contains the other.
Indeed, there are precisely $2^{k-1}$ blocks containing $q_i$ and precisely $2^{k-2}$ of these blocks contain $p$;
on the other hand, there are precisely $2^{k-2}$ blocks containing $q_i,t_i$ and each of them does not contain $p$;
this implies that every block containing $q_i$ contains precisely one of $p,t_i$. 
Therefore, every block contained in $P\setminus \{p\}$ is the union of $2^{k-2}$ distinct $\{q_i,t_i\}$.

In the case when all points of ${\mathcal D}$ are center points, 
${\mathcal D}$ is isomorphic to the design of points and hyperplane complements of ${\rm PG}(k-1,2)$ (Proposition \ref{prop-ps}).
Then ${\mathcal B}$ is a maximal singular subspace of ${\mathcal P}_{2^{k-2}}(2^k-1)$
and every such subspace is isomorphic to ${\rm PG}(k-1,2)$. 
For every $p\in P$ the subset ${\mathcal B}^p={\mathcal B}\setminus {\mathcal B}_p$ consisting of all blocks which do not contain $p$ is a hyperplane of ${\mathcal B}$
(since ${\mathcal B}^p$ contains precisely $2^{k}-1-2^{k-1}=2^{k-1}-1$ blocks and 
for any distinct blocks $B,B'\subset P\setminus \{p\}$ we have $B\triangle B'\subset P\setminus \{p\}$).

\section{One class of symmetric $(2^k-1, 2^{k-1}, 2^{k-2})$-designs with center blocks}

Let ${\mathcal D}=(P, {\mathcal B})$ be a symmetric $(2^k-1, 2^{k-1}, 2^{k-2})$-design with $k\ge 3$.
Then ${\mathcal B}$ is a $(2^k-1)$-element clique in the collinearity graph of ${\mathcal P}_{2^{k-2}}(2^k-1)$. 
We say that $O\in {\mathcal B}$ is a {\it center block}  if  for every $B\in {\mathcal B}$ distinct from $O$
the line of ${\mathcal P}_{2^{k-2}}(2^k-1)$ connecting $O$ and $B$ is contained in ${\mathcal B}$, in other words,
the symmetric difference $O\triangle B$ is a block of ${\mathcal D}$.
Every line of ${\mathcal P}_{2^{k-2}}(2^k-1)$ contained in ${\mathcal B}$ is a line of the dual design ${\mathcal D}^*$
and, consequently, 
every center block of ${\mathcal D}$ is a center point of ${\mathcal D}^*$.

\begin{exmp}{\rm
Let $O$ be a $2^{k-1}$-element subset of $P$.
The complement  $O^c=P\setminus O$ consists of $2^{k-1}-1$ elements. 
We take any $(2^{k-1}-1)$-element subset $Z\subset O$ and any symmetric $(2^{k-1}-1, 2^{k-2}, 2^{k-3})$-designs
$${\mathcal D}_O=(O^c, {\mathcal B}_O),\;\;{\mathcal D}_Z=(Z, {\mathcal B}_Z).$$
If $\delta:{\mathcal B}_O\to {\mathcal B}_Z$ is a bijection and 
${\mathcal B}$ is formed by $O$ and all 
$$X\cup \delta(X),\;\; X\cup (O\setminus \delta(X))\;\mbox{ with }\; X\in {\mathcal B}_O,$$
then ${\mathcal D}=(P, {\mathcal B})$ is a symmetric $(2^k-1, 2^{k-1}, 2^{k-2})$-design \cite[Proposition 2]{PPZ2}.
For every $X\in {\mathcal B}_O$ the blocks 
$$O, X\cup \delta(X), X\cup (O\setminus \delta(X))$$
form a line of ${\mathcal P}_{2^{k-2}}(2^k-1)$ which means that
$O$ is a center block of ${\mathcal D}$.
}\end{exmp}

As above, ${\mathcal D}=(P, {\mathcal B})$ is a symmetric $(2^k-1, 2^{k-1}, 2^{k-2})$-design and $k\ge 3$.
Suppose that $O$ is a center block of ${\mathcal D}$.
By \cite[Proposition 3]{PPZ2}, there is a unique symmetric $(2^{k-1}-1, 2^{k-2}, 2^{k-3})$-design $${\mathcal D}_O=(O^c, {\mathcal B}_O)$$
and for every $(2^{k-1}-1)$-element subset $Z\subset O$ there are a unique symmetric $(2^{k-1}-1, 2^{k-2}, 2^{k-3})$-design $${\mathcal D}_Z=(Z, {\mathcal B}_Z)$$
and a unique bijection $\delta_Z:{\mathcal B}_O\to {\mathcal B}_Z$ 
such that ${\mathcal B}$ consists of $O$ and all 
$$X\cup \delta_Z(X),\;\; X\cup (O\setminus \delta_Z(X))\;\mbox{ with }\; X\in {\mathcal B}_O.$$
For other $(2^{k-1}-1)$-element subset $Z'\subset O$ we have 
$$\delta_{Z'}(X)=\delta_Z(X)$$
if $\delta(X)$ is contained in $Z\cap Z'$ and 
$$\delta_{Z'}(X)=O\setminus\delta_Z(X)$$
otherwise.
In other words, ${\mathcal B}_Z\cap {\mathcal B}_{Z'}$ consists of all blocks of ${\mathcal D}_Z$ contained in $Z\cap Z'$
and for every $Y\in {\mathcal B}_Z\setminus {\mathcal B}_{Z'}$ the complement $O\setminus Y$ belongs to ${\mathcal B}_{Z'}\setminus {\mathcal B}_Z$.

If $k\in \{3,4\}$, then all symmetric $(2^{k-1}-1, 2^{k-2}, 2^{k-3})$-designs are isomorphic
and, consequently, 
${\mathcal D}_Z$ and ${\mathcal D}_{Z'}$ are isomorphic for 
any $(2^{k-1}-1)$-element subsets $Z,Z'\subset O$.
If $k\ge 5$, then  we are able to show the same  only in one special case.

\begin{prop}\label{prop-Z}
If for a certain $(2^{k-1}-1)$-element subset $Z\subset O$ the design ${\mathcal D}_Z$ is 
isomorphic to the design of points and hyperplane complements of ${\rm PG}(k-2,2)$,
then the same holds for any other $(2^{k-1}-1)$-element subset $Z'\subset O$.
\end{prop}

\begin{proof}
Let $Z$ and $Z'$ be distinct $(2^{k-1}-1)$-element subsets of $O$. 
If $p\in Z\setminus Z'$ and $p'\in Z'\setminus Z$, then
$$Z\cap Z'=O\setminus\{p,p'\}.$$
Recall that ${\mathcal B}_Z\cap {\mathcal B}_{Z'}$ consists of all blocks of ${\mathcal D}_Z$ contained in $Z\cap Z'$
and for every block $Y\in {\mathcal B}_Z$ containing $p$ the complement $O\setminus Y$ is a block of ${\mathcal D}_{Z'}$ containing $p'$. 

If ${\mathcal D}_Z$ is isomorphic to the designs of points and hyperplane complements of ${\rm PG}(k-2,2)$,
then every point of $Z$, in particular, $p$ is a center point of ${\mathcal D}_Z$. 
Therefore, 
$$Z\cap Z'=Z\setminus \{p\}=\{q_1,t_1,\dots,q_{2^{k-2}-1},t_{2^{k-2}-1}\},$$
where $\{p,q_i,t_i\}$ is a line of ${\mathcal D}_Z$ for every $i\in \{1,\dots,2^{k-2}-1\}$. 
By Section 2, 
every block of ${\mathcal D}_Z$ contained in $Z\cap Z'$ is the union of $2^{k-3}$ distinct $\{q_i,t_i\}$
and every block of ${\mathcal D}_Z$ containing $p$ intersects each $\{q_i,t_i\}$ precisely in one point. 
The permutation
$$\alpha_p=(p,p')(q_1,t_1)\dots(q_{2^{k-2}-1},t_{2^{k-2}-1})$$
preserves every block of ${\mathcal D}_Z$ contained in $Z\cap Z'$
and transposes every $Y\in {\mathcal B}_Z$ containing $p$ and $O\setminus Y$.
Therefore, $\alpha_p$ preserves ${\mathcal B}_Z\cap {\mathcal B}_{Z'}$ and transposes ${\mathcal B}_Z\setminus {\mathcal B}_{Z'}$ and 
${\mathcal B}_{Z'}\setminus {\mathcal B}_Z$, i.e. it induces an isomorphism between ${\mathcal D}_Z$ and  ${\mathcal D}_{Z'}$.
\end{proof}

From this moment, we assume that $O$ is a center block of ${\mathcal D}$ such that 
${\mathcal D}_Z$ is isomorphic to the design of points and hyperplane complements of ${\rm PG}(k-2,2)$
for every $(2^{k-1}-1)$-element subset $Z\subset O$ (see Proposition \ref{prop-Z}). 
In this case, we say that ${\mathcal D}$ is the {\it sum} of the symmetric $(2^{k-1}-1, 2^{k-2}, 2^{k-3})$-design  ${\mathcal D}_{O}$ and 
the design of points and hyperplane complements of ${\rm PG}(k-2,2)$. 
This sum depends  on the bijections $\delta_Z:{\mathcal B}_O\to {\mathcal B}_Z$ and each of these bijections determines the others.

Recall that $C^k_2$  is an abelian group generated by $k$ mutually commuting involutions. 
It contains precisely $2^k-1$ non-identity elements and each of them  is an involution. 
This group acts on ${\rm PG}(k-1,2)$ as the group of all reflections across hyperplanes.
There is a one-to-one correspondence $\alpha \leftrightarrow H_{\alpha}$
between non-identity elements of $C^k_2$ and hyperplanes of ${\rm PG}(k-1,2)$ such that 
for distinct involutions $\alpha,\beta\in C^k_2$ the hyperplane $H_{\alpha\beta}$ contains $H_{\alpha}\cap H_{\beta}$
and is distinct from the hyperplanes $H_{\alpha},H_{\beta}$.

We show that {\it the full automorphism group of ${\mathcal D}$ contains a subgroup $C^{k-1}_2$
which acts sharply transitively on $O$ and leaves every point of $O^c$ fixed.}

Let us take any $(2^{k-1}-1)$-element subset $Z\subset O$ and
the unique point $p'\in O$ which does not belong to $Z$. 
By our assumption, ${\mathcal B}_Z$ is a maximal singular subspace of ${\mathcal P}_{2^{k-3}}(2^{k-1}-1)$.
For every $p\in Z$ the subset ${\mathcal B}^p_Z$ formed by all blocks of ${\mathcal D}_Z$ which do not contain $p$ is a hyperplane of this subspace
(see Section 2). 
In the proof of Proposition \ref{prop-Z}, we constructed the permutation $\alpha_p$ which transposes ${\mathcal D}_Z$ and ${\mathcal D}_{Z'}$
for
$$Z'=(Z\setminus\{p\})\cup \{p'\}.$$
Observe that $\alpha_p$ is an involution which leaves each point of $O^c$ fixed.
It preserves every block belonging to ${\mathcal B}^p_Z={\mathcal B}_Z\cap {\mathcal B}_{Z'}$
and transposes every $Y\in {\mathcal B}_Z\setminus {\mathcal B}^p_Z$ and $O\setminus Y$.
This is an automorphism of ${\mathcal D}$ which acts on the set of blocks as follows: 
$\alpha_p$ preserves the center block $O$, for $X\in {\mathcal B}_O$ it preserves the blocks
$$X\cup \delta_{Z}(X)\;\mbox{ and }\;X\cup (O\setminus\delta_{Z}(X))$$ 
if $\delta_Z(X)$ belongs to ${\mathcal B}^p_Z$
and transposes them otherwise. 

For any distinct $p,q\in Z$  we take $t\in Z$ such that $\{p,q,t\}$ is a line of ${\mathcal D}_Z$. 
The hyperplane ${\mathcal B}^t_Z$ contains ${\mathcal B}^p_Z\cap {\mathcal B}^q_Z$ and is distinct from the hyperplanes ${\mathcal B}^p_Z,{\mathcal B}^q_Z$.
Indeed, if $Y\in {\mathcal B}_Z$ contains $t$, then it also contains $p$ or $q$ (see Section 2); 
therefore, if $p,q\not\in Y$, then $t\not\in Y$.
Since ${\mathcal B}_Z$ is the union of the hyperplanes ${\mathcal B}^p_Z,{\mathcal B}^q_Z,{\mathcal B}^t_Z$,
every $Y\in {\mathcal B}_Z$ is contained in precisely one or all of  them. 
A direct verification shows that 
$$\alpha_p\alpha_q(Y)=\alpha_q\alpha_p(Y)=\alpha_t(Y)$$ 
for all $Y\in {\mathcal B}_Z$. 
The automorphisms $\alpha_p\alpha_q, \alpha_q\alpha_p, \alpha_t$ leave every point of $O^c$ fixed.
So, they induce the same transformation of the set of blocks of ${\mathcal D}$ which  means that 
$$\alpha_p\alpha_q=\alpha_q\alpha_p=\alpha_t.$$ 
Therefore, all $\alpha_p$, $p\in Z$ (together with the identity) form a group $C^{k-1}_2$. 
We denote this group by $C^{k-1}_2(Z)$.
Since $\alpha_p$ is the unique element of this group transposing $p$ and $p'$, 
the action of $C^{k-1}_2(Z)$ on $O$ is sharply transitive. 

\begin{prop}\label{prop-C}
We have $C^{k-1}_2(Z)=C^{k-1}_2(Z')$
for any $(2^{k-1}-1)$-element subsets $Z,Z'\subset O$.
\end{prop}

\begin{proof}
Let $Z$ and $Z'$ be distinct $(2^{k-1}-1)$-element subsets of $O$.
The non-identity elements of $C^{k-1}_2(Z)$ and $C^{k-1}_2(Z')$ are $\{\alpha_p\}_{p\in Z}$ and $\{\alpha'_{p'}\}_{p'\in Z'}$, respectively.
If $f$ is an automorphism of ${\mathcal D}$ preserving $O$ and sending $Z$ to $Z'$,
then $f$ induces an isomorphism of ${\mathcal D}_Z$ to ${\mathcal D}_{Z'}$;
in particular, it transfers every ${\mathcal B}^p_Z$ to a certain ${\mathcal B}^{p'}_{Z'}$. 
A direct verification show that 
$$f\alpha_pf^{-1}(Y)=\alpha'_{p'}(Y)$$
for all $Y\in {\mathcal B}_{Z'}$. 
The automorphisms $\alpha_p,\alpha'_{p'}$ leave every element of $O^c$ fixed
and the same holds for $f\alpha_pf^{-1}$ (since $f$ preserves $O$ and, consequently, $O^c$). 
So, $f\alpha_pf^{-1}$ and $\alpha'_{p'}$ induce the same transformation of the set of blocks of ${\mathcal D}$
which means that 
$$f\alpha_pf^{-1}=\alpha'_{p'}.$$
Therefore, 
$$fC^{k-1}_2(Z)f^{-1}=C^{k-1}_2(Z').$$
Let us take $p\in Z\setminus Z'$. 
Then $\alpha_p$ transposes $Z,Z'$ and
$$\alpha_pC^{k-1}_2(Z)\alpha_p=C^{k-1}_2(Z').$$
This implies that $C^{k-1}_2(Z)=C^{k-1}_2(Z')$, since $\alpha_p$ is an element of $C^{k-1}_2(Z)$.
\end{proof}

So, $C^{k-1}_2(Z)$ does not depend on $Z$. 
For this reason, we will denote this group by $C^{k-1}_2(O)$. 

\begin{rem}{\rm
It was established in the proof of Proposition \ref{prop-C} that
$$fC^{k-1}_2(O)f^{-1}=C^{k-1}_2(O)$$
for every automorphism $f$ of ${\mathcal D}$ preserving $O$, i.e.
$C^{k-1}_2(O)$ is a normal subgroup in the group of all automorphisms of ${\mathcal D}$ preserving $O$.
}\end{rem}

\section{Main result}
Our main result states that there exist no other symmetric $(2^k-1, 2^{k-1}, 2^{k-2})$-designs with $k\ge 3$ which are
$(2^{k-1}-1)$-pyramidal over abelian groups.

\begin{theorem}\label{theorem-pyr}
Suppose that ${\mathcal D}$ is a symmetric $(2^k-1, 2^{k-1}, 2^{k-2})$-design with $k\ge 3$
which is $(2^{k-1}-1)$-pyramidal over an abelian group $G$, i.e.
there is an abelian group $G$ of automorphisms of ${\mathcal D}$  which leaves precisely $2^{k-1}-1$ points of ${\mathcal D}$ 
fixed and  acts sharply transitively on the remaining points.
Let $O$ be the complement of the set of fixed points of  $G$.
Then $O$ is a center block of ${\mathcal D}$,
the design ${\mathcal D}$  is the sum of a symmetric $(2^{k-1}-1, 2^{k-2}, 2^{k-3})$-design and 
the design of points and hyperplane complements of ${\rm PG}(k-2,2)$  and $G=C^{k-1}_2(O)$. 
\end{theorem}

\begin{rem}{\rm
For $k=2$ this statement is trivial. 
Every symmetric $(3,2,1)$-design is a $3$-element set with the $2$-element subsets as blocks. 
If a design automorphism leaves one of the three points fixed, then it transposes the remaining two. 
}\end{rem}

\section{Proof of Theorem \ref{theorem-pyr}}
Suppose that ${\mathcal D}$ is a symmetric $(2^k-1, 2^{k-1}, 2^{k-2})$-design with $k\ge 3$ and 
${\mathcal G}$ is an abelian group of automorphisms of ${\mathcal D}$ which leaves precisely $2^{k-1}-1$ points of ${\mathcal D}$ 
fixed and  acts sharply transitively on the remaining points.
Denote by $O$ the complement of the set of fixed points of  $G$.
Then  $|O|=2^{k-1}$.
Since $G$ acts on $O$ sharply transitively, $|G|=2^{k-1}$.

\begin{lemma}\label{lemma1}
The subset $O$ is a center block of ${\mathcal D}$.
The action of $G$ on the set of blocks of ${\mathcal D}$ is as follows: 
for every block $B\ne O$
an element of $G$ preserves $B$ and $O\triangle B$ or transposes these blocks. 
\end{lemma}

\begin{proof}
First, we show  that $|B\cap O|=2^{k-2}$ or $B=O$ for every block  $B$ of ${\mathcal D}$.
This implies that $O$ is a block of ${\mathcal D}$. Indeed, if this fails, then $O$ intersects every block of ${\mathcal D}$
precisely in a $2^{k-2}$-element subset which contradicts Fisher's inequality \cite{MM}.

Let $B$ be a block of ${\mathcal D}$.
If $|B\cap O|<2^{k-2}$, then $|B\cap O^c|>2^{k-2}$. For every $\alpha\in G$ we have 
\begin{equation}\label{eq}
\alpha(B)=(B\cap O^c)\cup\alpha(B\cap O)
\end{equation}
which means that 
$$B\cap O^c\subset B\cap \alpha(B)$$
and, consequently,  $B\cap \alpha(B)$ contains more than $2^{k-2}$ element. 
This means that $\alpha(B)=B$. So, every element of $G$ preserves $B$.
On the other hand, $B\cap O$ is non-empty (since $|O^c|=2^{k-1}-1$) and $G$ acts on $O$ transitively.
We get a contradiction which shows that $|B\cap O|\ge 2^{k-2}$.

Consider the case when $|B\cap O|>2^{k-2}$.
If $\alpha\in G$ and $\alpha(B)\ne B$, then there is a $2^{k-2}$-element  subset $Y\subset B\cap O$ such 
that $\alpha(Y)\subset O\setminus B$ which contradicts the fact that $|O\setminus B|<2^{k-2}$. 
As in the previous case, we obtain that every element of $G$ preserves $B$. 
This is impossible if $B\ne O$, since $G$ acts on $O$ transitively.
Therefore, $B=O$. 

So, $O$ is a block of ${\mathcal D}$. For every block $B\ne O$ there is $\alpha\in G$ such that $\alpha(B)\ne B$
(we take any $\alpha\in G$ sending a point of $B\cap O$ to a point of $O\setminus B$).
Since 
$$|B\cap \alpha(B)|=|B\cap O^c|=2^{k-2},$$ 
\eqref{eq} shows that  $B\cap O$ and $\alpha(B\cap O)$ are disjoint $2^{k-2}$-element subsets of $O$.
Then  $\alpha(B\cap O)$ is the complement of $B\cap O$ in $O$ which implies that  $O\triangle B =\alpha(B)$.
Therefore, $O$ is a center block of ${\mathcal D}$.
The second statement of Lemma \ref{lemma1} follows from the same arguments.
\end{proof}

\begin{lemma}\label{lemma2}
The group $G$ is $C^{k-1}_2$.
\end{lemma}

\begin{proof}
It is sufficient to show that $G$ is generated by involutions $\beta_0,\beta_1,\dots,\beta_{k-2}$ such that  
each  $\beta_i$ does not belong to the subgroup generated by $\beta_{i+1},\dots,\beta_{k-2}$.

{\it First step}.
Denote by ${\mathcal O}$ the set of $2^{k-2}$-element subsets of $O$
which are the intersections of $O$ with blocks of ${\mathcal D}$. 
For every $Y\in {\mathcal O}$ there is a unique block of ${\mathcal D}$ intersecting $O$ in $Y$.

For every $Y\in {\mathcal O}$ the complement $O\setminus Y$ also belongs to ${\mathcal O}$
(if $B$ is the block of ${\mathcal D}$ intersecting $O$ in $Y$, then $O\setminus Y$ is the intersection of $O$ and $O\triangle B$).

For any distinct $Y,Y'\in {\mathcal O}$ one of the following possibilities is realized:
$$Y'=O\setminus Y\;\mbox{ or }\; |Y\cap Y'|=2^{k-3}.$$
Suppose that $B$ and $B'$ are the blocks of ${\mathcal D}$ intersecting $O$ in $Y$ and $Y'$, respectively. 
If $Y'\ne O\setminus Y$, then $B'\ne O\triangle B$ which means that $B\cap O^c$ and $B'\cap O^c$
are distinct blocks of the symmetric $(2^{k-1}-1,2^{k-2},2^{k-3})$-design ${\mathcal D}_O$ (see Section 3). 
This implies that $|Y\cap Y'|=2^{k-3}$.

If $Y\in {\mathcal O}$, then every element of $G$ preserves $Y$ and $O\setminus Y$  or transposes them
(this follows from the second part of Lemma \ref{lemma1}). 
Furthermore, for every non-identity $\alpha\in G$ there is $Y\in {\mathcal O}$ such that $\alpha$ transposes $Y$ and $O\setminus Y$
(if this fails, then $\alpha(Y)=Y$ for all $Y\in {\mathcal O}$; 
since $\alpha$ preserves $O$ and leaves each point of $O^c$ fixed, 
it preserves all blocks of ${\mathcal D}$ which implies that $\alpha$ is identity).

{\it Second step}.
First, we show that $G$ contains an involution. 

Let us  take any $Y\in {\mathcal O}$ and any point $p\in Y$. 
If $\alpha\in G$ sends $p$ to a point of $Y$, then it preserves $Y$ and $O\setminus Y$;
if $\alpha$ sends $p$ to a point of $O\setminus Y$, then it transposes $Y$ and $O\setminus Y$.
Therefore, there are precisely $2^{k-2}$ elements of $G$ preserving $Y,O\setminus Y$
and the remaining $2^{k-2}$ elements of $G$ transpose these subsets. 
Let $\gamma_1,\dots, \gamma_{2^{k-2}}$ be the elements of $G$ which do not belong to
the subgroup of elements of $G$ preserving $Y,O\setminus Y$.
This subgroup contains $\gamma^2_i$  for every $i\in \{1,\dots, 2^{k-2}\}$. 
If every $\gamma_i$ is not an involution, or equivalently, every $\gamma^2_i$ is not identity,
then $\gamma^2_i=\gamma^2_j$ for some distinct $i,j$. Since $G$ is abelian, 
$$(\gamma_i\gamma^{-1}_j)^2=\gamma^2_i(\gamma^{-1}_j)^2=\gamma^2_i(\gamma^2_j)^{-1}=e.$$ 
Then $\gamma_i\gamma^{-1}_j$ is an involution as $\gamma_i\ne \gamma_j$.

So, $G$ contains an involution $\beta_0$.
By the first step, there is $Y_1\in {\mathcal O}$ such that  $\beta_0$ transposes $Y_1$ and $O\setminus Y_1$.
Let $G_1$ be the subgroup of elements of $G$ preserving $Y_1$. 
Then $G$ is generated by $G_1$ and $\beta_0$
(every $\alpha\in G\setminus G_1$ transposes $Y_1$ and $O\setminus Y_1$ and $\beta_0\alpha$ belongs to $G_1$).

Every $Y\in {\mathcal O}$ distinct from $Y_1,O\setminus Y_1$
intersects $Y_1$ precisely in a $2^{k-3}$-element subset. 
We take any such $Y$ and consider the $2^{k-3}$-element subsets $Y_1\cap Y$ and $Y_1\setminus Y$.
As at the beginning of this step, we establish that  $G_1$ contains an involution $\beta_1$. 
There is $Y'_2\in {\mathcal O}$ such that $\beta_1$ transposes $Y'_2$ and $O\setminus Y'_2$.
Since $\beta_1$ preserves $Y_1,O\setminus Y_1$, the subset $Y'_2$ is distinct from $Y_1,O\setminus Y_1$ 
and 
$$Y_2=Y_1\cap Y'_2$$
consists of $2^{k-3}$ points.

Let $G_2$ be the subgroup of elements of $G_1$ preserving $Y_2$. 
If $\alpha\in G_1\setminus G_2$, then it preserves $Y_1$ and transposes $Y'_2,O\setminus Y'_2$
which means that $\alpha$ transposes $Y_2,Y_1\setminus Y_2$ and $\beta_1\alpha$ belongs to $G_2$.
Therefore, $G_1$  is generated by $G_2$ and $\beta_1$.

If $k=3$, then $Y_2$ is a one-element subset and the group $G_2$ is trivial
which means that $G$ is generated by $\beta_0,\beta_1$ and, consequently, it is $C^2_2$.

{\it Third step}.
Suppose that  $k\ge 4$ and there are 
$$Y_1,Y'_2,\dots, Y'_l\in {\mathcal O},\;\;\; 2\le l\le k-1$$ 
such that every
$$Y_i=Y_1\cap Y'_2\cap\dots\cap Y'_i,\;\;\;2\le i\le l$$ 
consists of $2^{k-i-1}$ points. 
Suppose also that there are groups 
$$G=G_0\supset G_1\supset \dots \supset G_l$$
and involutions  
$$\beta_0\in G_0,\beta_1\in G_1,\dots, \beta_{l-1}\in G_{l-1}$$
such that $G_i$ is formed by all elements of $G_{i-1}$ preserving $Y_i$
and  $G_{i-1}$ is generated by $G_i$ and $\beta_{i-1}$ for every $i\in \{1,\dots, l\}$.

If $l=k-1$, then the involutions $\beta_0,\beta_1,\dots,\beta_{k-2}$ are as required.

Suppose that $l\le k-2$.
Then $Y_l$ contains at least two points. 
For any two distinct point of $Y_l$ there is a block of ${\mathcal D}$ containing precisely one of them.
The intersection of this block and $O$ is an element of ${\mathcal O}$ which intersects $Y_l$ in a non-empty proper subset.
Show that this intersection contains  precisely  $2^{k-l-2}$ points.
Consider distinct $p,q\in Y_l$ and $Y\in {\mathcal O}$ which contains $p$ and does not contain $q$.
We take $\alpha\in G$ sending $p$ to $q$. 
It transposes $Y$ and $O\setminus Y$ (since $Y$ contains $p$ and does not contain $q$). 
Also, $\alpha$ preserves $Y_1,Y'_2,\dots, Y'_l$ (since each of them contains both $p,q$)
and, consequently, $\alpha$ preserves $Y_l$. 
This means that $\alpha$ transposes $Y_l\cap Y$ and $Y_l\setminus Y$. 
Therefore, these subsets contain the same number of elements. 
Since $|Y_l|=2^{k-l-1}$, we obtain that $|Y_l\cap Y|=2^{k-l-2}$.

Using $Y_l\cap Y$ and $Y_l\setminus Y$, we establish that $G_l$ contains an involution $\beta_l$ as in the second step.

There is $Y'_{l+1}\in {\mathcal O}$ such that $\beta_l$ transposes $Y'_{l+1}$ and $O\setminus Y'_{l+1}$. 
If $Y_l$ is contained in $Y'_{l+1}$ or $O\setminus Y'_{l+1}$, then $\beta_l$ preserves each of these subsets
(since it preserves $Y_l$) which is impossible. 
Therefore, $Y'_{l+1}$ intersects $Y_l$ in a non-empty proper subsets and, consequently,
$$Y_{l+1}=Y_l\cap Y'_{l+1}$$
consists of $2^{k-l-2}$ points.
Let $G_{l+1}$ be the subgroup of elements of $G_l$ preserving $Y_{l+1}$.
As in the second step, we obtain that $G_l$ is generated by $G_{l+1}$ and $\beta_l$. 

So, we can construct the required $k-1$ involutions recursively.
\end{proof}

Every non-identity $\alpha \in G$ is an involution. 
Since there is $Y\in {\mathcal O}$ such that $\alpha$ transposes $Y,O\setminus Y$, 
it  is the composition of $2^{k-2}$ mutually commuting transpositions.

Let $Z$ be a $(2^{k-1}-1)$-element subset of $O$. 
We need to show that the design ${\mathcal D}_Z$ is isomorphic to the design of points and hyperplane complements of ${\rm PG}(k-2,2)$ 
(see Section 3).
We take any $p\in Z$ and suppose that $p'$ is the unique point of $O$ which does not belong to $Z$. 
Consider $\varepsilon_p\in G$ transposing $p$ and $p'$. 
We have
$$\varepsilon_p=(p,p')(q_1,t_1)\dots(q_{2^{k-2}-1},t_{2^{k-2}-1}),$$
where $\{p,p',t_1,q_1,\dots,q_{2^{k-2}-1},t_{2^{k-2}-1}\}=O$.

This automorphism preserves a block $Y$ of ${\mathcal D}_Z$ or sends it to $O\setminus Y$.
If this block contains $p$, then $\varepsilon_p(Y)=O\setminus Y$
(since $p'\in O\setminus Y$).
This means that every block of ${\mathcal D}_Z$ containing $p$ intersects each $\{q_i,t_i\}$ precisely in one point. 
Therefore, for every $i\in \{1,\dots, 2^{k-2}-1\}$ there is no block of ${\mathcal D}_Z$
containing $p,q_i,t_i$, in other words, $\{p,q_i,t_i\}$ is a line of ${\mathcal D}_Z$.
Since $p$ is arbitrary taken, we obtain that any two distinct points of ${\mathcal D}_Z$ are collinear.
By Proposition \ref{prop-ps}, ${\mathcal D}_Z$ is isomorphic to the design of points and hyperplane complements of ${\rm PG}(k-2,2)$.
Then $\varepsilon_p$ coincides with the automorphism $\alpha_p$ considered in Section 3 and $G=C^{k-1}_2(O)$.

\end{document}